\documentclass[11pt]{article}

\usepackage{fullpage}
\usepackage[T1]{fontenc}
\usepackage{graphicx}
\usepackage{color}
\usepackage[numbers]{natbib}
\usepackage[vlined,boxed]{algorithm2e}
\usepackage{todonotes}
\usepackage{authblk}
\usepackage{subcaption}
\usepackage{enumitem,units}
\usepackage{mathtools}
\usepackage{amsfonts}
\usepackage{amsmath,amssymb,mathrsfs}
\usepackage{amsthm}
\usepackage{abstract}
\usepackage{adjustbox}
\usepackage[title]{appendix}
\usepackage{caption}
\usepackage{tikz}
\usetikzlibrary{matrix,decorations.pathreplacing,calc,fit,backgrounds}
\usetikzlibrary{positioning, automata}
\usetikzlibrary{arrows.meta}
\usetikzlibrary{fit}
\usepackage{hyperref}
\hypersetup{colorlinks,
pdfstartview = FitH,
bookmarksopen = true,
bookmarksnumbered = true,
linkcolor = blue,
plainpages = false,
hypertexnames = false,
urlcolor = blue,
citecolor = blue}

\usepackage{thmtools} 
\usepackage{thm-restate}
\usepackage{cleveref}
\usepackage{multicol}

\newtheorem{theorem}{Theorem}
\newtheorem{lemma}{Lemma}

\newtheorem{corollary}{Corollary }

\definecolor{c1}{rgb}{1.0, 0.03, 0.0}
\definecolor{c2}{rgb}{0.0, 0.0, 1.0}
\definecolor{c3}{RGB}{0,150,0}
\definecolor{c4}{rgb}{0.5, 0, 0.5}
\definecolor{c5}{rgb}{1, 0, 1}
\definecolor{c6}{rgb}{1, 0.5, 0}
\definecolor{c7}{rgb}{0, 0.7, 0.7}
\definecolor{c8}{RGB}{139,69,19}
\colorlet{cr}{c6}
\colorlet{cl}{c7}

\tikzset{%
    vertex/.style={%
        circle,inner sep=1,fill=gray!40, draw=black 
    }
}
\tikzset{%
    edge/.style={%
        black, ultra thick
    }
}
\tikzset{%
    uncoloredEdge/.style={%
        black, thick
    }
}
\tikzset{%
    arrow/.style={%
        black, ultra thick
    }
}
\tikzset{%
    arrow2/.style={%
        black, ultra thick, ->
    }
}

\tikzset{
  toprule/.style={%
    execute at end cell={%
      \draw [line cap=rect,#1] (\tikzmatrixname-\the\pgfmatrixcurrentrow-\the\pgfmatrixcurrentcolumn.north west) -- (\tikzmatrixname-\the\pgfmatrixcurrentrow-\the\pgfmatrixcurrentcolumn.north east);%
    }
  }
}

\tikzset{
    ncbar angle/.initial=90,
    ncbar/.style={
        to path=(\tikztostart)
        -- ($(\tikztostart)!#1!\pgfkeysvalueof{/tikz/ncbar angle}:(\tikztotarget)$)
        -- ($(\tikztotarget)!($(\tikztostart)!#1!\pgfkeysvalueof{/tikz/ncbar angle}:(\tikztotarget)$)!\pgfkeysvalueof{/tikz/ncbar angle}:(\tikztostart)$)
        -- (\tikztotarget)
    },
    ncbar/.default=0.1cm,
}

\tikzset{square left brace/.style={ncbar=0.1cm}}
\tikzset{square right brace/.style={ncbar=-0.1cm}}
\tikzset{round left paren/.style={ncbar=0.1cm,out=107,in=-107}}
\tikzset{round left parenbigger/.style={ncbar=0.1cm,out=100,in=-100}}
\tikzset{round left parenlarge/.style={ncbar=0.1cm,out=95,in=-95}}
\tikzset{round right paren/.style={ncbar=0.1cm,out=73,in=-73}}
\tikzset{round right parenbigger/.style={ncbar=0.1cm,out=80,in=-80}}
\tikzset{round right parenlarge/.style={ncbar=0.1cm,out=85,in=-85}}

\newcommand{\N}{\mathbb{N}}
\newcommand{\R}{\mathbb{R}}

\newif\ifpaper
\papertrue
\renewenvironment{proof}{{\textit{Proof.}}}{\qed} 

\SetKwFunction{MinimumStarColoring}{Minimum\-Star\-Coloring}
\SetKwFunction{EnsureFeasibility}{Ensure\-Feasibility}
\SetKwFunction{ColorOneEdge}{Color\-One\-Edge}
\SetKwFunction{MinimumStarColoringFlow}{Minimum\-Star\-Coloring\-Flow}

\title{On the Min-Max Star Partitioning Number}
\date{}
\author[1]{Sarah Feldmann \footnote{funded by the Deutsche Forschungsgemeinschaft (DFG, German Research Foundation) - 314838170, GRK 2297 MathCoRe}}
\author[2]{Torben Schürenberg \footnote{funded by the Deutsche Forschungsgemeinschaft (DFG) - Project number 281474342.}}
\affil[1]{Otto-von-Guericke University Magdeburg, Germany }
\affil[2]{University of Bremen, Germany}

\begin{document}  
\maketitle
\vspace{-0.5cm}
\begin{abstract}
\noindent In this paper, we introduce a novel star partitioning problem for simple connected graphs $G=(V,E)$. The goal is to find a partition of the edges into stars that minimizes the maximum number of stars a node is contained in while simultaneously satisfying node-specific capacities. We design and analyze an efficient polynomial time algorithm with a runtime of $\mathcal{O}(|E|^2)$ that determines an optimal partition. Moreover, we explicitly provide a closed form of an optimal value for some graph classes. 
We generalize our algorithm to find even an optimal star partition of linear hypergraphs, multigraphs, and graphs with self-loop. We use flow techniques to design an algorithm for the star partitioning problem with an improved runtime of $\mathcal{O}(\log(\Delta) \cdot |E| \cdot \min\{|V|^{\frac{2}{3}},|E|^{\frac{1}{2}}\})$, where $\Delta$ is maximum node degree in $G$. In contrast to the unweighted setting, we show that a node-weighted decision variant of this problem is \texttt{strongly NP-complete} even without capacity constraints. Furthermore, we provide an extensive comparison to the problem of minimizing the minimum indegree satisfying node capacity constraints.

\end{abstract}
\section{Introduction} \label{ch:introduction} 
We introduce a novel graph problem that arises in static communication networks, the \textit{min-max star partitioning problem} (\textsc{MinMaxStar}). The goal is 
to find a partition of the edges into stars that minimizes the maximum number of stars a node is contained in. A star 
is a tree, where at most one node has more than one neighbor. In our setting, a star is required to contain at least one edge.
The \textit{internal node} of a star containing more than one edge is defined as the node with at least two neighbors. For stars consisting of only one edge, we fix one of the incident nodes as the internal node. Additionally, each node can be the internal node of at most one star.
The star partitioning problem can be seen as a new kind of edge coloring problem. 

In this paper, we define a set of valid edge colorings and optimize a min-max objective function. 
We initially deal with simple connected graphs $G=(V,E)$, where each edge contains two nodes. In Section \ref{sec:extensions_other_classes} we extend the problem to graphs with self-loops, multigraphs, and linear hypergraphs. Every node $v$ has a given capacity restriction, denoted by $\kappa_v$.
We initially color every node in a distinct color. A valid edge coloring is achieved when each edge has a color identical to that of one of its incident nodes, and for every node $v$ the number of different colors of incident edges to $v$ is at most $\kappa_v$. It is possible that no valid edge coloring exists, that fulfills the capacity restrictions. 
Our goal is to minimize the maximum number of incident edge colors of a node in valid edge colorings.
This graph coloring defines a partition of the edges of our graph into star graphs. That is why we refer to it as a \textit{star coloring} or a \textit{star partition}.
Depending on the coloring, the maximum number of different edge colors to which a node in the graph is incident may vary. Our goal is to minimize this value among all valid star colorings, which we refer to as $x^*(G)$.

\;\newline
\noindent\textbf{Our results.}
We explicitly provide a closed form of an optimal value for some graph classes in Section \ref{sec:optValue} and construct a simple polynomial-time algorithm based on depth-first search, which finds an optimal star partition in $\mathcal{O}(|E|^2)$ time in Section \ref{ch:algorithm}. We apply this algorithm to linear hypergraphs.
Given an integer $x$, we test if $x^*(G)\leq x$ holds by computing a maximum flow in a network with unit capacities constructed from $G$. 
This flow technique and a binary search allow us to improve the simple algorithm to a runtime of $\mathcal{O}(\log(\Delta) \cdot |E| \cdot \min\{|V|^{\frac{2}{3}},|E|^{\frac{1}{2}}\})$ in Section \ref{ch:improvedAlgo}. 
We apply both algorithms to graphs with multiple edges and self-loops in Section \ref{sec:extensions_other_classes}.
We show similarities and differences between \textsc{MinMaxStar} and the problem of \textit{minimizing the maximum indegree} (\textsc{MinMaxInd}), which was studied by Asahiro et al. in \cite{DBLP:journals/ijfcs/AsahiroJMO11, DBLP:journals/ijfcs/AsahiroMOZ07} and Venkateswaran in \cite{DBLP:journals/dam/Venkateswaran04}. Additionally, we provide solutions that are optimal for both problems simultaneously when no capacity constraints are present in Section \ref{sec:StarInd}.
As a separation, we show in Section \ref{ch:weightedMinMax} that the node-weighted decision variants \textsc{W-MinMaxStar} and \textsc{W-MinMaxInd} of these problems are \texttt{strongly NP-complete} even without capacity constraints and provide a 2-approximation algorithm for \textsc{W-MinMaxInd} and a 4-approximation algorithm for \textsc{W-MinMaxStar} in Section \ref{ch:approx}.

\;\newline
\noindent\textbf{Related work.}
Various star coloring and star partitioning problems with different objective functions have been explored in the literature. 

Graph colorings have been researched extensively in the past, for both edge and node colorings \cite{10.5555/1457583, Jensen1994GraphCP}.
Feasible edge colorings are subject to different rules and various objective functions are studied in the literature. The most well-known edge coloring rule specifies that adjacent edges can not share the same color. The edge chromatic number is defined as the minimum number of different colors needed to color a graph according to this rule, researched for instance in \cite{DBLP:journals/dm/BeinekeW73, DBLP:journals/dm/FaudreeS84, DBLP:journals/jgt/MelnikovV99, DBLP:journals/dm/Zhang16b}. In his seminal work \cite{vizing1964estimate}, Vizing proved that the edge chromatic number of any simple graph $G$ is either $\Delta(G)$ or $\Delta(G)+1$, where $\Delta(G)$ denotes the maximum degree of the graph.

Divya and Vijayakumar analyze the partition of nodes of split graphs into as few sets as possible, such that the induced subgraph of every set in the partition is a star \cite{Divya_Vijayakumar}.
The star chromatic number researched by Borodin \cite{DBLP:journals/dm/Borodin79} is the minimum number of colors needed for a vertex coloring in which every path on four vertices uses at least three distinct colors. 
Fertin et al. define a star coloring as a node coloring where incident nodes have different colors and no path of length three is bicolored \cite{fertin2004star}. Egawa et al. \cite{egawa1997star} and Shalu et al. \cite{shalu2022induced} define a star partition as a partition of the nodes of a given graph such that for each node set of the partition, the induced subgraph is a star. In contrast to our definition, Egawa et al. and Shalu et al. color or partition the nodes and not the edges.
Minimizing the total number of stars results in the vertex cover problem, for which Karp shows the \texttt{NP-completeness} of the decision variant in \cite{Karp1972}. 

Given a star partitioning, we define the value of a node as the number of stars the node is contained in. Another natural utilitarian objective function is given by minimizing the sum of these values. This objective counts every edge and every star exactly once, and is thus a reformulation of the \texttt{NP-complete} vertex cover problem. 
On the other hand, if for every node $v$ we do not count the star with internal node $v$, the problem gets easy to solve because the optimal objective value $|E|$ is independent of the edge partitioning.

\begin{figure}[tb]
\centering
\subfloat[An optimal solution of \textsc{MinMaxStar} which is not optimal for \textsc{MinMaxInd}]{
\trimbox{-2cm 0cm -2cm 0cm}{
\begin{tikzpicture}[scale=1.3]
\coordinate (c1) at (-1,0);
\coordinate (c2) at (1,0);
\coordinate (c3) at (0,1);

\foreach \i in {1,2,3} {
\node[vertex,fill=c\i!50] (v\i) at (c\i) {$v_{\i}$};
}

\draw [->,ultra thick,c1] (v1) to (v2);
\draw [->,ultra thick,c3] (v3) to (v2);
\draw [->,ultra thick,c3] (v3) to (v1);
\end{tikzpicture}
}}\hfill
\subfloat[An optimal solution of \textsc{MinMaxInd} which is not optimal for \textsc{MinMaxStar}]{
\trimbox{-2cm 0cm -2cm 0cm}{
\begin{tikzpicture}[scale=1.3]
\coordinate (c1) at (0,0);
\coordinate (c2) at (1,0.5);
\coordinate (c3) at (0,1);
\coordinate (c4) at (2,0);
\coordinate (c5) at (2,1);

\foreach \i in {1,2,3,4,5} {
\node[vertex,fill=c\i!50] (v\i) at (c\i) {$v_{\i}$};
}
\draw [->,ultra thick,c3] (v3) to (v1);
\draw [->,ultra thick,c2] (v2) to (v1);
\draw [->,ultra thick,c3] (v3) to (v2);
\draw [->,ultra thick,c4] (v4) to (v2);
\draw [->,ultra thick,c2] (v2) to (v5);
\draw [->,ultra thick,c4] (v4) to (v5);
\end{tikzpicture}
}}
\setlength{\belowcaptionskip}{-0.1cm}
\caption{
Examples where optimal solutions of \textsc{MinMaxInd} and \textsc{MinMaxStar} do not coincide.
}
\label{fig:example_orientation2}
\end{figure}
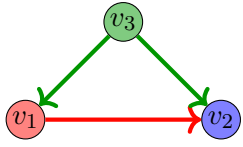
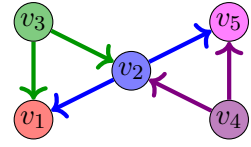 
In a graph where each node has a distinct color, directing an edge is equivalent to assigning the color of its tail to the edge. Many papers discussing various edge orientation problems were recently published, for instance, by Asahiro et al. \cite{asahiroLongestPathOrientation}.
Other orientation problems have been studied by Ahadi and Dehghan \cite{DBLP:journals/ipl/AhadiD13}, Asahiro et al. \cite{DBLP:journals/ijfcs/AsahiroJMO11, DBLP:journals/ijfcs/AsahiroMOZ07}, Borradaile et al. \cite{DBLP:journals/jgaa/BorradaileIMOWZ17} and Venkateswaran \cite{
DBLP:journals/dam/Venkateswaran04}. Borradaile et al. introduce an algorithm that minimizes the lexicographic order of the indegrees. 
Venkateswaran provides an algorithm with a runtime of $\mathcal{O}(|E|^2)$, that computes an edge orientation which minimizes the maximum indegree $k^*(G)$. This was further improved by Asahiro et al. \cite{DBLP:journals/ijfcs/AsahiroMOZ07} to a runtime of $\mathcal{O}(|E|^{\frac{3}{2}}\cdot \log(\Delta(G)))$ using flow techniques.
Without node capacity constraints in contrast to our objective, Asahiro et al. and Venkateswaran direct the edges to minimize only the maximum number of outgoing and incoming edges, respectively, over all nodes. This corresponds to minimizing the maximum number of stars containing $v$ that do not have $v$ as an internal node. Thus, $k^*(G) \leq x^*(G) \leq k^*(G)+1$. We will denote the problem of finding an orientation that achieves $k^*(G)$ while ensuring node capacity constraints by \textsc{MinMaxInd}.  
As evident in Figure \ref{fig:example_orientation2}, optimal star colorings do not always coincide with optimal edge orientations. 
Therefore, our variant needs to be examined separately. 

\;\newline
\noindent\textbf{Application example.} We are given a set of players, represented by the nodes of a graph. For each player, the task is to communicate with all their neighboring players in this static communication network. Each player $i$ has a secret communication key with the weight $w_i$. If two neighboring players want to communicate with each other, one of the two players must remember the key of the other. The sum of the weights of the communication keys that the players need to memorize is called the memory requirement of a player. What is the smallest maximum memory requirement over all players? 
If each person does not need to memorize their own communication key, we have \textsc{W-MinMaxInd}. If a player needs to memorize their own communication key if it is memorized by at least one other player, we have the problem \textsc{W-MinMaxStar}. If we set the weight of each key to 1, we obtain the unweighted problems \textsc{MinMaxInd} and \textsc{MinMaxStar}, for which we can construct a solution that is optimal for both cases simultaneously. 

\clearpage\section{Problem statement}\label{ch:model}
We consider a non-empty undirected simple connected graph $G=(V, E)$. The terms \textit{coloring} and \textit{partitioning} are used interchangeably.
For each node ${v\in V}$, we define the degree $\delta(v) \coloneqq {| \{ e\in E \;|\; v\in e\}|} $ as the number of edges incident to $v$. The degree of the graph $\Delta(G)$ is thus the maximum degree of all nodes $\Delta(G) \coloneqq \max_{v \in V} \delta(v)$. 
A node capacity $\kappa_v\in \N_{\geq 0}$ and a unique color $C(v)$ are assigned to each node $v$.
If no node capacities are given, they are assumed to be $\Delta(G)$. An edge coloring assigns a color $C(e)$ to each edge $e$. A star coloring is an edge coloring where each edge $e\in E$ has a node $v \in e$ with $C(e)= C(v)$. A \textit{valid} star coloring on $G$ is a star coloring where the number of distinct incident edge colors never exceeds the node capacity $\kappa_v$ for all $v\in V$. A partial star coloring on $G$ is a star coloring of $G'=(V,E')$ for some $E'\subseteq E$.
The \textit{set of all valid star colorings} on $G$ is referred to as $\mathcal{C}(G)$, and the set of all valid partial star colorings on $G$ as $\mathcal{C}'(G)$. We omit $G$ if it is obvious. 
For each $C\in \mathcal{C}(G)$ and $v\in V$, we define $x(G,C,v) \coloneq |\{ C(e) \;|\; e\in E, v\in e\}|$, the number of different edge colors that the node $v$ is incident to. Additionally, we define the \textit{star partitioning number} $x(G,C)\coloneq \max\limits_{v\in V} \;x(G,C,v)$, the maximum value of $x(G,C,v)$ over all nodes $v\in V$. The \textit{min-max star partitioning number} in the set of all valid star colorings is denoted by $$x^*(G)\coloneq \min\limits_{C\in \mathcal{C}(G)}\;x(G, C) = \min\limits_{C\in \mathcal{C}(G)}\;  \max\limits_{v\in V} \; x(G, C,v).$$
If no valid star coloring exists, we set $x^*(G) \coloneq \infty$. 

We define the \textit{min-max star partitioning problem} (\textsc{MinMaxStar}) 
and similar to Asahiro et al. in \cite{DBLP:journals/ijfcs/AsahiroMOZ07} the \textit{min-max indegree problem} (\textsc{MinMaxInd}) with node capacities as follows:\\

\begin{tabular}{l}
\hspace{-0.5cm}\textbf{(\textsc{MinMaxStar}): min-max star partitioning problem} \\
\textbf{Input:} A simple connected graph $G=(V,E)$ and $\kappa_v\in\N_{\geq 0}$ for all $v\in V$.\\
\textbf{Output:} Compute a valid star coloring $C$ with $x(G,C)=x^*(G)$ or state
that none exists.
\end{tabular}
\vspace{0.3cm}

\begin{tabular}{l}
\hspace{-0.5cm}\textbf{(\textsc{MinMaxInd}): min-max indegree problem} \\
\textbf{Input:} A simple connected graph $G=(V,E)$ and $\kappa_v\in\N_{\geq 0}$ for all $v\in V$.\\
\textbf{Output:} Compute an orientation, which minimizes the maximum indegree
such that \\
$\delta^-(v)\leq \kappa_v$ for all $v\in V$ or state that none exists.
\end{tabular}
\vspace{0.3cm}

\noindent 
A partial coloring $C'\in \mathcal{C}'(G)$ is called $x$-satisfying if every node is incident to at most $x$ different edge colors in every valid star coloring $C\in \mathcal{C}(G)$, which can be constructed from $C'$ by assigning a color to each uncolored edge.
Given a coloring $C\in\mathcal{C}(G)$, the only way to reduce the number of stars a node is contained in is to color more incident edges in its color.
For a given $x$, a partial coloring $C'$ is thus an $x$-satisfying partial coloring if for every node $v\in V$ it holds that every completion $C\in \mathcal{C}(G)$ assigns the color $C(v)$ to at least 
$l(v,x)\coloneq \max \{ 0,\delta(v)-\min \{ \kappa_v,x\}+1 \}$ many edges incident to $v$ if $l(v,x)>1$ holds. In this setting $x^*(G)$ is the smallest $x$ for which it is possible, that every node $v$ with $l(v,x)>1$ has at least $l(v,x)$ incident edges colored in its color.

A valid star coloring $C$ can be transformed into an orientation $\Lambda(G,C)$ by directing each edge $\{v,w\}$ colored with $C(v)$ from $v$ to $w$. 
Additionally, each orientation $\Lambda(G,C)$ can be transformed into a valid star coloring $C$ by coloring each edge $\{v,w\}$, which is directed to $w$ with the color $C(v)$. Thus, valid colorings and orientations are inverse to each other. So we can use these perspectives interchangeably.

\section{Optimal value for special graph classes}\label{sec:optValue}
For a few graph classes, we can state the optimal value $x^*(G)$ explicitly.
\begin{restatable}{lemma}{lemmaSpecialGraphs}
\label{lemma:special-graphs}
The following statements hold for simple graphs $G$.
\begin{enumerate}
    \item[(a)] ${x^*(G) \geq 1}$ iff $G$ has at least one edge. 
    \item[(b)] ${x^*(G) \leq 1}$ iff $G$ is a cycle free graph with a diameter at most two. 
    \item[(c)] If $G$ is a pseudoforest, i.e., in each connected component there exists at most one cycle, then ${x^*(G) \leq 2}$.
    \item[(d)] Let $G$ be the fully connected bipartite graph $K_{n,n}$ with $1<n\in \mathbb{N}$. Then ${x^*(K_{n,n})=\left\lceil \frac{n}{2}\right\rceil + 1}$. 
\end{enumerate}
\end{restatable}
\noindent The proof is given in Appendix \ref{app:characterisation}.
\section{Algorithm for \textsc{MinMaxStar}}\label{ch:algorithm}
In this section, we sketch the algorithm \MinimumStarColoring, which computes a min-max star partition or states infeasibility in $\mathcal{O}(|E|^2)$ time. The pseudocode can be found in the appendix in Algorithm \ref{algorithm}.

A valid partial star coloring $C$ can be seen as a partially directed graph $G_C$, where edges $\{v,w\}$ with the color $C(v)$ are directed toward $v$. Note that $\Lambda(G,C)$ is the inverse orientation of $G_C$. We use this formulation for the algorithm.

We always know that $x^*(G)\leq \Delta(G)$ if the instance is feasible. Thus, we start with $x=\Delta(G)$. If we find an $x$-satisfying partial coloring $C$ for some $x$, we decrease $x$ by one and repeat the paragraph below using the partial coloring $C$ until we find an integer $x$ with $x^*(G)>x$. Completing $C$ with random valid colors yields a min-max star partition of $G$.

We aim to find a star coloring $C$ fulfilling $x(G,C)\leq x$ or determine $x^*(G)>x$ the following way. 
For each node $v$ with $l(v,x)>1$, we try to color incident edges with $C(v)$ until $l(v,x)$ edges are colored with $C(v)$. This coloring is done via a depth-first search in $G_C$ for an uncolored edge. If we find such an edge, we insert it into the path and invert the direction of all edges on that path, such that the whole path points toward $v$. The coloring $C$ is changed appropriately. 
For all nodes $w$ along the path, except $v$, inverting the direction of the path does not change the number $|C(w)|$. Additionally, we enlarge $|C(v)|$ by one.
If such a path does not exist, we return that $x^*(G)>x$ holds.

Figure \ref{fig:Color} shows the last steps of the algorithm using an example, where it has already found a partial 3-satisfying coloring and tries to improve it to a partial 2-satisfying coloring. Note that the depth-first search fails even though there exists an uncolored edge in the graph.

Each depth-first search runs in $\mathcal{O}(|E|)$ time (see \cite{even2011graph}) and either colors an additional edge or states infeasibility. During the algorithm, we store for each node $v$ its degree and the number of edges colored with $C(v)$. Thus, we only need $\mathcal{O}(|E|^2)$ steps.

If the depth-first search fails at a node $v$, where less than $l(v,x)>1$ edges are colored with $C(v)$, then there is no uncolored edge in $G_C$ reachable from $v$. During the algorithm, we colored for each node $w$ at most $l(w,x)$ edges with $C(w)$ iff $l(w,x)>1$ holds. Let $K$ be the set of nodes in $G_C$ reachable from $v$. The number of edges incident to at least one node in $K$ is strictly less than $\sum_{w \in K : l(w,x)>1} l(w,x)$. The strict inequality results from $l(v,x)>1$ and because we colored strictly less than $l(v,x)$ edges with $C(v)$. Therefore, it is not possible to color an additional edge incident to any node $u\in K$ with $C(u)$ without reducing the number of edges colored with $C(w)$ for another node $w \in K$. 
Each $u\in K$ is reachable from $v$ and the set of nodes reachable from $u\in K$ is contained in $K$. Thus, we have $x^*(G)>x$. This results in the following theorem.

\begin{restatable}{theorem}{equadrat}
\label{theorem:correctness}
The  \MinimumStarColoring algorithm terminates after $\mathcal{O}(|E|^2)$ steps with the correct solution to \textsc{MinMaxStar}.
\end{restatable}

\begin{figure}[tb]
\centering
\subfloat[]{
\begin{tikzpicture}[scale=1.3]
\coordinate (c1) at (0,-0.5);
\coordinate (c2) at (0,0.5);
\coordinate (c3) at (1,0);
\coordinate (c4) at (3,-0.5);
\coordinate (c5) at (3,0.5);
\coordinate (c6) at (2.3,0.1);
\coordinate (c7) at (-1,0);

\foreach \i in {1,2,3,4,5,6,7} {
\node[vertex,fill=c\i!50] (v\i) at (c\i) {$v_{\i}$};
}
\draw[color=c2] (c2) circle (2.8mm); 

\draw[arrow2,color=c1] (v1)--(v2);
\draw[arrow2,color=c1] (v1)--(v3);
\draw[arrow2,color=c1] (v1)--(v4);
\draw[uncoloredEdge] (v1)--(v7);
\draw[arrow2,color=c3] (v3)--(v2);
\draw[arrow2,color=c2] (v2)--(v5);
\draw[arrow2,color=c2] (v2)--(v7);
\draw[arrow2,color=c3] (v3)--(v4);
\draw[arrow2,color=c3] (v3)--(v6);
\draw[uncoloredEdge] (v4)--(v5);
\draw[uncoloredEdge] (v5)--(v6);
\end{tikzpicture}
}
\qquad
\subfloat[]{
\begin{tikzpicture}[scale=1.3]
\coordinate (c1) at (0,-0.5);
\coordinate (c2) at (0,0.5);
\coordinate (c3) at (1,0);
\coordinate (c4) at (3,-0.5);
\coordinate (c5) at (3,0.5);
\coordinate (c6) at (2.3,0.1);
\coordinate (c7) at (-1,0);

\foreach \i in {1,2,3,4,5,6,7} {
\node[vertex,fill=c\i!50] (v\i) at (c\i) {$v_{\i}$};
}

\draw[arrow2,color=c1,] (v1)--(v2);
\draw[arrow2,color=c3] (v3)--(v1);
\draw[arrow2,color=c1] (v1)--(v4);
\draw[arrow2,color=c1] (v1)--(v7);
\draw[arrow2,color=c2] (v2)--(v3);
\draw[arrow2,color=c2] (v2)--(v5);
\draw[arrow2,color=c2] (v2)--(v7);
\draw[arrow2,color=c3] (v3)--(v4);
\draw[arrow2,color=c3] (v3)--(v6);
\draw[arrow2,color=c4] (v4)--(v5);
\draw[uncoloredEdge] (v5)--(v6);
\draw[color=c4] (c4) circle (2.8mm); 
\end{tikzpicture}
}
    \caption{The depicted graph is the reverse graph of $G_C$. For $x=2$ the depth-first search starting at $v_2$ is successful along the path $(v_2, v_3,v_1, v_7)$ in (a). In (b) the depth-first search starting at $v_4$ succeeds only once and then fails. This results in $x^*(G)=3$. Note that there is still one uncolored edge in the graph.}
    \label{fig:Color}
\end{figure}
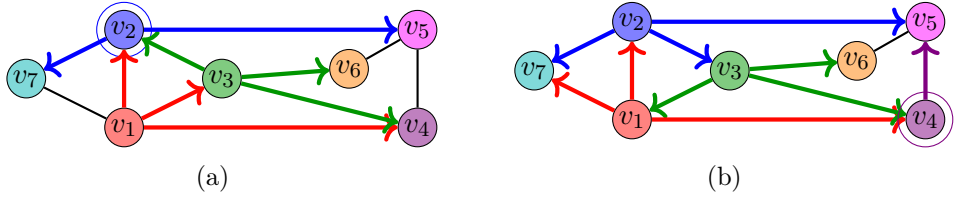

\section{Extension to other graph classes}\label{sec:extensions_other_classes}
\noindent\textbf{Self-loops.}\label{section:self-loops} 
By definition, no self-loop can be part of a star. To include self-loops in the edge partitioning, we change the definition of star graphs and allow the internal node of a star to have a self-loop. 
If we define the degree of a node as the number of incident edges, then a self-loop only increases the degree of the node by one. The algorithms and correctness proofs can be applied directly to this case. 

\;\newline
\noindent\textbf{Multigraphs.}\label{section:multigraphs}
Coloring parallel edges in a multigraph with the same color does not increase the objective value. Therefore, there are always optimal colorings that color all parallel edges with the same color. 
Thus, we transform each multigraph $H_M$ into a simple graph $G=(V,E)$ by merging each set of parallel edges into one edge, searching for a coloring in $G$ instead, and coloring every edge in $H_M$ with the assigned color of its representative.

\;\newline
\noindent\textbf{Hypergraphs.}\label{section:hypergraphs}
A hypergraph $H=(V, E)$ is called \textit{linear} if for every edge pair $e_1 \neq e_2\in E$ we have $| e_1 \cap e_2 | \leq 1$.
The depth-first search from the previous section can be applied directly to linear hypergraphs with a runtime of $\mathcal{O}(|E| \cdot \eta(H))$, where $\eta(H):=\max\{|e|  :\; e \in E\}$ is the maximum cardinality of the edges. This is due to an edge $e$ now visited not at most twice but at most $|e|$ times.
Therefore, the total runtime of the \MinimumStarColoring algorithm is $\mathcal{O}(\eta(H) \cdot |E|^2)$.
If $\eta(H)$ is bounded by a constant, we get the runtime of $\mathcal{O}(|E|^2)$.
As shown in Appendix \ref{section:hypergraphs}, the \MinimumStarColoring algorithm can not be applied directly to general hypergraphs.

\section{Improved algorithm}\label{ch:improvedAlgo}
For simple connected graphs $G=(V,E)$ with node capacities, we use the flow techniques of Asahiro et al. \cite{DBLP:journals/ijfcs/AsahiroMOZ07} to improve the runtime from $\mathcal{O}(|E|^2)$ to $\mathcal{O}(\log(\Delta) \cdot |E| \cdot \min\{|V|^{\frac{2}{3}},|E|^{\frac{1}{2}}\})$.
As noted in Section \ref{ch:model}, valid star colorings $C$ and orientations $\Lambda(G,C)$ can be transformed naturally into each other such that the transformation is inverse.
For a given $x \in \N_{\geq 0}$, we test if $x^*(G)\leq x$ holds, adapting the flow construction introduced by Asahiro et al. in \cite{DBLP:journals/ijfcs/AsahiroMOZ07}.
For any orientation $\Lambda(G,C)$, any number $x \in \N_{\geq 0}$ and any capacity constraint $\kappa_v$ we use the outdegree $\delta^+(v)$ of a node $v$ to define its slackness as
\begin{displaymath}
    s(G,\Lambda(G,C),v,x)\coloneqq
\begin{cases}
\delta^+(v)& \hspace{-1cm} \text{if } l(v,x) \leq 1,\\ 
\begin{matrix*}[l]
\delta^+(v)-l(v,x)\\
 \ = \delta^+(v)- \max \{ 0,\delta(v)-\min \{ \kappa_v,x\}+1 \}
\end{matrix*} & \text{else}.
\end{cases}
\end{displaymath}
If the slackness at $v$ is positive, then it is possible to change the orientation of $|s(G,\Lambda(G,C), v,x)|$ edges that are directed outward from $v$ without violating \mbox{$x(G,C,v)\leq x$} and the capacity constraint $\kappa_v$ of the induced coloring. On the other hand, if the slackness of $v$ is negative, then it is necessary to reverse the orientation of at least $|s(G,C,v,x)|$ edges, which are directed to $v$ in order to satisfy $x(G,C,v)\leq x$ and the capacity constraint $\kappa_v$.

Using our graph with an arbitrary orientation $\Lambda$, we construct a flow multi\-graph $\Tilde{G}=(\Tilde{V}= V\cup\{s,t\},\Tilde{E})$.
Similar to Asahiro et al. \cite{DBLP:journals/ijfcs/AsahiroMOZ07}, the edge set $\Tilde{E}$ is constructed by adding the following edges to $E$.
For each node $v$ with $s(G,\Lambda(G,C),v,x)>0$ we add $|s(G,\Lambda(G,C),v,x)|$ times the edge $(s,v)$ and for each node $v$ with $s(G,\Lambda(G,C),v,x)<0$ we add  $|s(G,\Lambda(G,C),v,x)|$ times the edge $(v,t)$. 
Because the absolute value of the slackness is bounded by the degree of the node, we get $|\tilde{E}|\leq 3|E|$. Finding the maximum flow in the unit-capacity network $\Tilde{G}$ can be done in $\mathcal{O}(|\tilde{E}| \cdot \min\{|\Tilde{V}|^{\frac{2}{3}},|\tilde{E}|^{\frac{1}{2}}\})$ time, as Even and Tarjan show in \cite{doi:10.1137/0204043}. By construction, this flow has an integer value at each edge. 
Note that this maximum flow algorithm can be replaced by any other algorithm, where a solution has the same properties. The runtime has to be adapted accordingly. 

Asking if there exists an $s$-$t$ flow in $\Tilde{G}$ with a value of at least $\sum_{v\in V} \max\{0, -s(G,\Lambda(G,C),v,x)\}$ is equivalent to the problem of increasing the slackness of all nodes with negative slackness to at least $0$. This is the case iff this flow fully uses every edge connected to $t$. 
If such an integer flow exists, we flip the orientation of every edge $e\in E$ with flow equal to one to get an induced $x$-satisfying coloring. Examples can be found in Appendix \ref{app:improvedAlgo}. 

Using this construction and binary search, we find $x \in \{1,\cdots,\Delta(G),\infty\}$ with $x=x^*(G)$. 
We get a runtime of $\mathcal{O}(\log(\Delta) \cdot |E| \cdot \min\{|V|^{\frac{2}{3}},|E|^{\frac{1}{2}}\})$. The correctness of this approach can be proved analogously to \cite{DBLP:journals/ijfcs/AsahiroMOZ07}. We call this algorithm \MinimumStarColoringFlow. The key insight about the improved runtime is, that the slackness of all nodes with negative slackness is increased simultaneously. 
Note that this approach can not be applied to linear hypergraphs directly, due to the lack of the flow techniques for hypergraphs.

\begin{theorem}
\label{thm:reduction_edge_orientation}
\MinimumStarColoringFlow computes $x^*(G)$ for a simple connect\-ed non-empty graph $G$ with a runtime of $\mathcal{O}(\log(\Delta) \cdot |E| \cdot \min\{|V|^{\frac{2}{3}},|E|^{\frac{1}{2}}\})$.
\end{theorem}
\noindent\MinimumStarColoringFlow can also be used to compute an optimal solution of \textsc{MinMaxInd} in $\mathcal{O}(\log(\Delta) \cdot  |E| \cdot \min\{|V|^{\frac{2}{3}},|E|^{\frac{1}{2}}\})$ time by reducing it in linear time to \textsc{MinMaxStar}, as we show in Theorem \ref{thm:reduction_MMIPTOMMSP}. Therefore, this algorithm is a generalization of the algorithm which Asahiro et al. present in \cite{DBLP:journals/ijfcs/AsahiroMOZ07}.

\section{Comparing \textsc{MinMaxStar} and \textsc{MinMaxInd}}\label{sec:StarInd}
In this section, we show a linear reduction from \textsc{MinMaxInd} to \textsc{MinMaxStar}.
Afterward, we prove, that there always exists an orientation of $G$ that is optimal for \textsc{MinMaxStar} and \textsc{MinMaxInd} simultaneously when we consider the case without node capacities, i.e., $\kappa_v\geq \delta_v$ for all $v\in V$. Furthermore, we state explicitly how to compute such a solution.

\begin{restatable}{theorem}{reductionMMIPTOMMSP}
\label{thm:reduction_MMIPTOMMSP}
There exists a linear-time reduction from \textsc{MinMaxInd} to \textsc{MinMaxStar} for simple graphs.
\end{restatable}
\noindent
\begin{proof}{ (Sketch)}
Given a graph $G=(V,E)$ with node capacities $\kappa$, for which we want to solve \textsc{MinMaxInd}, we construct a graph $G'$ with node capacities $\kappa'$, on which we solve \textsc{MinMaxStar} instead.
The graph $G'=(V \cup V', E \cup E')$ is constructed as follows: For each node $v$, we insert a copy $v'$ into $V'$ and insert the edge $\{v,v'\}$ into $E'$. The node $v$ gets the capacity $\kappa'_v := \kappa_v+1$ and $v'$ has the capacity $\kappa'_{v'}=1$. This construction is visualized in Figure \ref{fig:reduction_MMIP_to_MMSP}. The proof of $k^*(G)=x^*(G')-1$ is given in Appendix \ref{app:secReductions}.
\end{proof}

\begin{theorem}
    For any given simple connected graph $G$ without node capacity constraints, we can compute a solution that is optimal for \textsc{MinMaxStar} and \textsc{MinMaxInd} simultaneously in $\mathcal{O}(\log(\Delta) \cdot |E| \cdot \min\{|V|^{\frac{2}{3}},|E|^{\frac{1}{2}}\})$.
\end{theorem}
\noindent\begin{proof}
    Using the algorithm of Theorem \ref{thm:reduction_edge_orientation} we compute an optimal orientation $\Lambda_1$ for \textsc{MinMaxStar} with value $x^*(G)$. Combining Theorem \ref{thm:reduction_edge_orientation} and the linear reduction of Theorem \ref{thm:reduction_MMIPTOMMSP} we get an optimal orientation $\Lambda_2$ for \textsc{MinMaxInd} with value $k^*(G)$. On the one hand, if $x^*(G) = k^*(G) + 1$, then $\Lambda_2$ is also optimal for \textsc{MinMaxStar}. If on the other hand, $x^*(G) = k^*(G)$ then $\Lambda_1$ is also optimal for \textsc{MinMaxInd}, since in $\Lambda_1$ the indegree of any node is at most $x^*(G)$. 
\end{proof}

\;\newline
\noindent
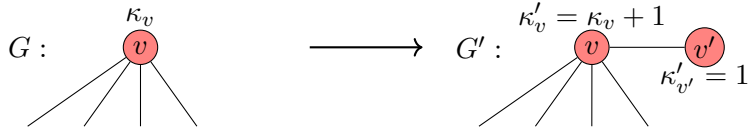
\begin{figure}[tb]
    \centering
\begin{tikzpicture}[scale=1.5]
\tikzset{%
    vertex/.style={%
        circle,inner sep=1,fill=gray!40, draw=black, minimum width=13pt
    }
}

\newcommand\xdiff{4};

\coordinate (c1) at (0,0);
\coordinate (c2) at (\xdiff,0);
\coordinate (c3) at (\xdiff+1,0);
\draw[->,thick] (1.5,0) to (\xdiff-1.5,0);
\node[] at (-1,0){$G:$};
\node[] at (-1+\xdiff,0){$G':$};

\node[vertex,fill=c1!50] (v) at (c1) {$v$};
\node[vertex,fill=c1!50] (w) at (c2) {$v$};
\node[vertex,fill=c1!50] (wt) at (c3) {$v'$};
\draw[-] (w) to (wt);

\foreach \i in {-1,-0.5,0,0.5}{
\draw[-] (v) to (\i,-0.7);
\draw[-] (w) to (\i+\xdiff,-0.7);
}

\node[]  [above of=v, node distance=4mm] {$\kappa_v$};
\node[]  [above of=w, node distance=4mm] {$\kappa'_v=\kappa_v+1$};
\node[]  [below of=wt, node distance=4mm] {$\kappa'_{v'}=1$};

\end{tikzpicture}
    \caption{A node $v$ in $G$ and its representation in $G'$}
    \label{fig:reduction_MMIP_to_MMSP}
\end{figure}
\vspace{-1cm}
\section{Hardness results}\label{ch:weightedMinMax}
In this section, we provide several reductions between the node-weighted decision versions of \textsc{MinMaxStar}, \textsc{MinMaxInd} and bin packing. For each node $v$, we introduce a positive node weight $w_v$. In contrast to the unweighted versions, we show that the weighted versions of both problems are \texttt{strongly NP-complete}. Note that this holds without any capacity constraints on the nodes.

For a valid star partition $C$, we redefine $x(G,C,v)$ for each node $v\in V$ as the sum of the weights of all stars $v$ is contained in. The weight of a star is defined as the weight of the node coloring all edges in the star. This node is defined as the internal node. For a given orientation $\Lambda$, we define $k(G,\Lambda,v)$ for each node $v\in V$ as the sum of weights of nodes having a directed edge toward $v$. If the graph and the orientation are clear from context, we simply write $x(v)$ or $k(v)$.
\newline

\begin{tabular}{l}
\hspace{-0.5cm}\textbf{(\textsc{W-MinMaxStar}):} node-weighted min-max star partitioning problem \\
\textbf{Input:} $k\in \R_{> 0}$, a simple graph $G=(V,E)$, weights $w_v \in \R_{>0}$  $\forall v\in V$. \\
\textbf{Question:}  Does there exist a star partitioning such that $\max\limits_{v\in V}\;x(v)\leq k$?
\end{tabular}

\begin{tabular}{l}
\hspace{-0.5cm}\textbf{(\textsc{W-MinMaxInd}):} node-weighted min-max indegree problem \\
\textbf{Input:} $k\in \R_{> 0}$, a simple graph $G=(V,E)$, weights $w_v \in \R_{>0}$  $\forall v\in V$. \\
\textbf{Question:} Does there exist an orientation such that $\max\limits_{v\in V}\;k(v)\leq k$?
\end{tabular}

\subsection{Reduction from \textsc{W-MinMaxInd} to \textsc{W-MinMaxStar}}\label{sec:redwIndwStar}
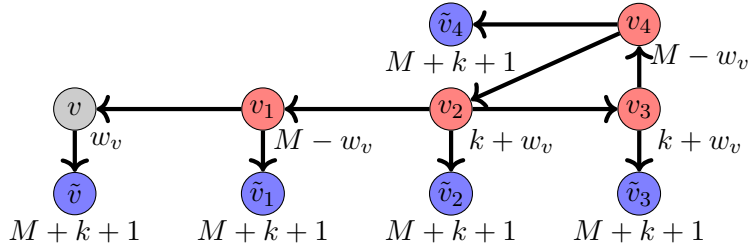
\begin{figure}[tb]
    \centering
\begin{tikzpicture}[scale=2.5]
\tikzset{%
    vertex/.style={%
        circle,inner sep=1,fill=gray!40, draw=black, minimum width=16pt
    }
}
\newcommand\size{0.45};
\coordinate (cv) at (0,0);
\coordinate (cb) at (0,-\size);
\coordinate (c1) at (1,0);
\coordinate (c2) at (2,0);
\coordinate (c3) at (3,0);
\coordinate (c4) at (3,\size);
\coordinate (cb1) at (1,-\size);
\coordinate (cb2) at (2,-\size);
\coordinate (cb3) at (3,-\size);
\coordinate (cb4) at (2,\size);

\node[vertex] (v) at (cv) {$v$};
\node[]  [below right of=v, node distance=6mm] {$w_v$};
\node[vertex,fill=c2!50] (vb) at (cb) {$\Tilde{v}$};
\node[]  [below of=vb, node distance=5mm] {$M+k+1$};

\foreach \i/\j in {1/{M-w_v},2/{k+w_v},3/{k+w_v},4/{M-w_v}} {
\node[vertex,fill=c1!50] (v\i) at (c\i) {$v_{\i}$};
\node[]  [below right of=v\i, node distance=6mm] {$\quad \quad \j$};
}

\foreach \i in {1,2,3,4} {
\node[vertex,fill=c2!50] (vb\i) at (cb\i) {$\Tilde{v}_{\i}$};
\node[]  [below of=vb\i, node distance=5mm] {$M+k+1$};
}
\draw [arrow2] (v1) to (v);
\draw [arrow2] (v) to (vb);
\draw [arrow2] (v1) to (vb1);
\draw [arrow2] (v2) to (v1);
\draw [arrow2] (v2) to (vb2);
\draw [arrow2] (v2) to (v3);
\draw [arrow2] (v4) to (v2);
\draw [arrow2] (v3) to (vb3);
\draw [arrow2] (v3) to (v4);
\draw [arrow2] (v4) to (vb4);

\end{tikzpicture}
\setlength{\belowcaptionskip}{-0.5cm}
\caption{Visualization of the gadget $G_{v}$ with a valid orientation with the value of at most $M+k$. The gadget forces $v$ to orientate $\{v,v_1\}$ toward $v$ and $\{v,\Tilde{v}\}$ toward $\Tilde{v}$ in each star partition with the value of at most $M+k$. The weights of the nodes are visualized next to each node.}
    \label{fig:reductionWeightedVariants}
\end{figure}

\noindent In this subsection, we will reduce \textsc{W-MinMaxInd} to \textsc{W-MinMaxStar}. We show that deciding \textsc{W-MinMaxInd} on $G$ is equivalent to deciding \textsc{W-MinMax\-Star} on a constructed graph $G'$ with value at most $M+k$ with $M:= k + 2 \max_{v \in V} w_v $.
The graph $G'$ is constructed by adding the gadget $G_v$ to every node $v\in V$ where $G_v=(V_v\cup \Tilde{V_{v}},E_v)$ with the node sets $V_v=\{v,v_1,v_2,v_3,v_4\}$ and $\Tilde{V_{v}}= \bigcup_{w \in V_v} \Tilde{w}$ and the edge set
$ E_v = \big\lbrace \{v,v_1\}, \{v_1,v_2\}, \{v_2,v_3\}, \{v_2,v_4\}, \{v_3,v_4\}\big\rbrace $ $\cup \bigcup_{w \in V_v} \big\lbrace \{w,\Tilde{w}\} \big\rbrace$
as visualized in Figure \ref{fig:reductionWeightedVariants}.
The weights of the nodes in $G_v$ are according to this figure.
\begin{restatable}{theorem}{reductionWeightedCases}
\label{thm:reductionWeightedCases}
There is a linear-time reduction from \textsc{W-MinMaxInd} to \textsc{W-MinMaxStar}.
\end{restatable}
\noindent The proof is given in Appendix \ref{app:secReductions}. Note that because of the node weights, the reduction needs to be more evolved than Theorem \ref{thm:reduction_MMIPTOMMSP}.

\subsection{Reduction from bin packing to \textsc{W-MinMaxInd}}
In this subsection, we will reduce the \texttt{strongly NP-complete} bin packing problem to \textsc{W-MinMax\-Ind}. This shows together with Theorem \ref{thm:reductionWeightedCases} that \textsc{W-MinMax\-Ind} and \textsc{W-MinMaxStar} are \texttt{strongly NP-complete}.
\begin{theorem}
    \textsc{W-MinMaxInd} is \texttt{strongly NP-hard}.
\end{theorem}
\begin{proof}
    We construct a reduction from the \texttt{strongly NP-hard} bin packing problem \cite{10.5555/574848}. 
    Let a bin packing problem $\mathcal{B}$ be given by a finite set $I={i_1,\ldots, i_n}$ with size $s_i\in \mathbb{Z}^+$ for every element $i\in I$, which should be distributed into $K \geq 2$ bins with bin capacities $c$. 
    We construct a graph $G_{\mathcal{B}}=(V=V_1 \dot\cup V_2, E)$, where we introduce for every element of $i_j \in I$ a node $v_j\in V_1$ and for every bin $b_j$ a node $b_j \in V_2$. We construct an edge $\{v,b\}$ for every $v\in V_1, b\in V_2$. Additionally, we set the node weights $w_v=s_v$ for every $v\in V_1$ and $w_b= \frac{c}{K-1}$ for every $b\in V_2$. The graph $G_{\mathcal{B}}$ and its node weights have a size depending polynomial on the size of the bin packing problem $\mathcal{B}$.

    Given a partition of the elements in $I$ into $K$ bins, which obeys the capacity restriction $c$ of the bins, we construct an orientation of $G_{\mathcal{B}}$ with a maximum weighted indegree of at most $c$ the following way.
    If $i_j \in b$, we orient the edge $\{v_j,b\}$ toward $b$ and the edges $\{v_j,b'\}$ with $b' \in V_2 \setminus \{b\}$ toward $v_j$.
    Each node $v \in V_1$ has exactly $K-1$ incoming edges, which weights sum up to $c$.
    The weights of the elements in bin $b$ sum up to at most $c$, and the only edges directed toward the corresponding node $b \in V_2$ are edges from nodes which are constructed from these elements.

    Given an orientation $\Lambda$ of $G_{\mathcal{B}}$ with a maximum weighted indegree of at most $c$, we construct a partition of $I$ into $K$ bins with capacity $c$, as visualized in Figure \ref{fig:red_binPacking}.
    For every $v\in V_1$, there exists at least one edge that is directed away from $v$ since otherwise $k(G,\Lambda,v)=K\cdot \frac{c}{K-1} > c$. 
    For each node $v_j\in V_1$, let $b_l \in V_2$ be the node with the smallest index, where the edge $\{v_j,b_l\}$ is directed toward $b_l$. We put the element $i_j$ into the bin $b_j$. The weight of the elements in the bin $b$ is at most the sum of the weights of the incoming edges. This weight is at most $c$.
\end{proof}

\begin{figure}[tb]
    \centering
\begin{tikzpicture}[scale=2]
\tikzset{ 
    vertex/.style={
        circle,inner sep=1,fill=gray!40, draw=black, minimum width=13pt
    }
}

\coordinate (c1) at (1,0);
\coordinate (c2) at (2,0);
\coordinate (c3) at (3,0);
\coordinate (c4) at (4,0);
\coordinate (c5) at (5,0);
\coordinate (c6) at (6,0);

\coordinate (cb1) at (2,-1);
\coordinate (cb2) at (3.5,-1);
\coordinate (cb3) at (5,-1);

\node[] at (0.5,0) {$V_1:$};
\node[] at (0.5,-1) {$V_2:$};

\foreach \i/\g in {1/1,2/1,3/3,4/6,5/8,6/9}{
\node[vertex, fill=c\i!50] (v\i) at (c\i) {$v_{\i}$};
\node[]  [above of=v\i, node distance=5mm] {$\g$};
}

\foreach \i in {1,2,3}{
\node[vertex] (vb\i) at (cb\i) {$b_{\i}$};
\node[]  [below of=vb\i, node distance=5mm] {$5$};
}

\foreach \i/\j in {b1/2,b1/5,b1/6, b2/3,b2/4,b2/6, b3/2,b3/3,b3/4,b3/5}{
\draw[->] (v\i) to (v\j);
}

\foreach \i/\j in {1/b1,3/b1,4/b1, 1/b2,2/b2,5/b2, 1/b3,6/b3}{
\draw[edge,->,c\i] (v\i) to (v\j);
}

\end{tikzpicture}
    \caption{Visualization of an orientation with value at most $c$ of the constructed graph $G_{\mathcal{B}}$ for a given bin packing instance $\mathcal{B}$ with $K=3$ and $c=10$. Items are assigned to the leftmost bin in case the item has more than one outgoing edge.}
    \label{fig:red_binPacking}
\end{figure}
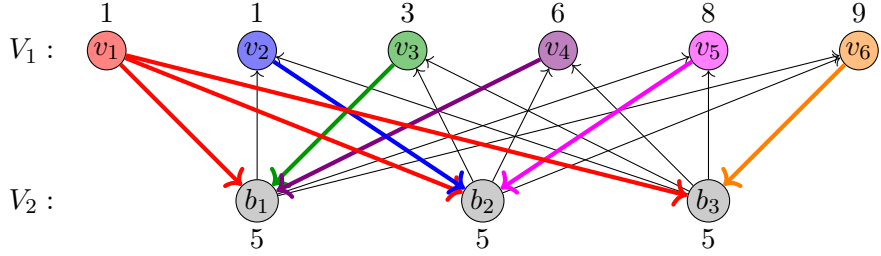

\begin{corollary}\label{thm:hardness}
Both \textsc{W-MinMaxInd} and \textsc{W-MinMaxStar} are \texttt{strongly} \texttt{NP\-complete}. 
\end{corollary}
\noindent The \texttt{strong NP-hardness} of the problem is a consequence of the polynomial-time reductions from bin packing to \textsc{W-MinMaxInd} and from \textsc{W-MinMaxInd} to \textsc{W-MinMaxStar}.
\textsc{W-MinMaxInd} and \textsc{W-MinMaxStar} are in \texttt{NP}, because given an orientation $\Lambda$ or a valid coloring $C$, we can verify in polynomial time if $k(G,\Lambda,v) \leq x$ holds for every node $v$ or if $x(G,C,v) \leq x$ holds for every node $v$.

\vspace{-0.1cm}
\section{Approximation algorithms}\label{ch:approx}
\vspace{-0.1cm}
We have seen that \textsc{W-MinMaxInd} and \textsc{W-MinMaxStar} are \texttt{NP-complete}. We construct approximation algorithms for finding minimal values $k$ for which these problems are YES-instances for a given graph $G$.

\begin{theorem}\label{thm:approx}
    Each algorithm that is an $\alpha$-approximation for \textsc{W-MinMaxInd} is a $2\alpha$-approximation for \textsc{W-MinMaxStar}.
\end{theorem}
\noindent\begin{proof}
    Note that $k^*\leq x^*$ holds for optimal solutions $k^*$ of \textsc{W-MinMaxInd} and $x^*$ of \textsc{W-MinMaxStar}. 
    Let $\Lambda$ be an orientation with $k(G,\Lambda) \leq \alpha \cdot k^*$. For each node $v\in V$ we can distinguish two cases. 
    If $v$ has no incoming edges, then $x(G,\Lambda,v)=k(G,\Lambda,v)$. Else, we have $x(G,\Lambda,v)=k(G,\Lambda,v)+w_v$ and the value $w_v$ contributes to $k(G,\Lambda,w)$ for at least one $w\in V \setminus \{v\}$. Therefore, we can deduce that if $v$ has at least one outgoing edge then $w_v \leq k(G,\Lambda)$. Thus, it holds $ x(G,\Lambda) \leq 2 \cdot k(G,\Lambda) \leq 2 \alpha \cdot k^* \leq 2 \alpha \cdot x^*$.
\end{proof}

\;\\
\noindent
In scheduling unrelated parallel machines there are $m$ parallel machines and $n$ independent jobs. 
The objective is to assign each job to exactly one machine in a manner that minimizes the makespan. Assigning a job $i$ to a machine $j$ contributes to the processing time of machine $j$ with $p_{i,j}$.
A special instance occurs when taking an \textsc{W-MinMaxInd} instance and interpreting each node as a machine and each edge $e=\{v_i,v_j\}$ as a job with processing times $p_{e,v_i}=w_{v_j}$, $p_{e,v_j}=w_{v_i}$ and $p_{e,v}=\infty$ for $v \not \in e$.

Thus, \textsc{W-MinMaxInd} is a special case of scheduling unrelated parallel machines for which Lenstra et al. showed a 2-approximation in \cite{DBLP:journals/mp/LenstraST90} and their approximation algorithm yields a 4-approxi\-mation for \textsc{W-MinMaxStar}.
\vspace{-0.1cm}
\section{Conclusion and Outlook}\label{ch:furtherQuestions}
\vspace{-0.1cm}
We introduced \textsc{MinMaxStar} and created two algorithms finding an optimal solution in $\mathcal{O}(|E|^2)$ and $\mathcal{O}(\log(\Delta) \cdot |E| \cdot \min\{|V|^{\frac{2}{3}},|E|^{\frac{1}{2}}\})$ time. We then extended \textsc{MinMaxStar} to multigraphs and linear hypergraphs. 
We showed the connection between \textsc{MinMaxStar} and \textsc{MinMaxInd}.
Afterward, we introduced node-weighted variants \textsc{W-MinMaxStar} and \textsc{W-MinMaxInd}, which we proved to be \texttt{strongly NP-complete} in contrast to  \textsc{MinMaxStar} and \textsc{MinMaxInd}. Furthermore, we provided approximation algorithms for \textsc{W-MinMaxStar} and \textsc{W-MinMaxInd} with approximation factors of 4 and 2, respectively.

In the future, one could investigate the problem of maximizing the minimum number of incident edge colors of nodes, also known as the Santa Claus problem in scheduling \cite{10.1145/1132516.1132522}.
Further research can be done to find an algorithm that computes a star partition of hypergraphs, which are not linear. 
One may investigate if there exist even better approximation algorithms or approximation lower bounds for \textsc{W-MinMaxStar} and \textsc{W-MinMaxInd}. Furthermore, one can construct a star partitioning problem with weighted edges and different ways of incorporating the edge weights into the objective. 
This would be similar to graph balancing \cite{ebenlendr2014graph, DBLP:conf/icalp/JansenR19, DBLP:journals/scheduling/VerschaeW14}.

\;\newline
\noindent\textbf{Acknowledgements}\label{ch:Acknowledgements}
Constructive feedback by 
Tom Freudenberg, Volker Kaibel, 
Alexander Lindmayr, 
Elias Pitschmann,
Sebastian Sager,
Jens Schlöter, 
Daniel Schmand, 
Nicole Schröder
and our anonymous reviewers
is gratefully acknowledged.

\clearpage
\bibliographystyle{plainnat}
\bibliography{main.bib}

\newpage
\appendix
\section{Appendix}\label{app}
\subsection{Omitted proof of Section \ref{ch:model}}\label{app:characterisation}
\lemmaSpecialGraphs*
\noindent
\begin{proof}\renewcommand{\qed}{}
\begin{enumerate}
    \item[(a)] Is trivial.
    \item[(b)] Let $x^*(G)\leq 1$. Assume that there is a cycle of length $\ell>2$ in $G$ or that $G$ has a path of length at least three. Then there exist at least two neighboring nodes $v$ and $w$ with degree at least two. Thus, $v$ and $w$ color all their incident edges with their color and the edge $\{v,w\}$ is colored with two colors, which is not feasible.
    
    Let $G$ be a cycle free graph with a diameter at most two. Then each path in $G$ has length at most two.
    A valid coloring $C$ with $x(G,C)\leq 1$ can be constructed by first coloring all edges on paths of length two with the color of the middle node. Each remaining edge is colored with the color of one of its incident nodes.
    
    \item[(c)]  We prove the statement for each connected component $G_i$ of $G$.
    If $G_i$ contains no cycle, then it is a tree. We select a root node and recursively color all edges connecting the node $v$ with the children of $v$ with the color $C(v)$. This coloring is a valid star coloring, where every node is incident to at most two different edge colors. Consequently, it can be deduced that $x^*(G_i)\leq 2$.
    
    If $G_i$ contains one cycle, then by removing an edge from the cycle, this component becomes a tree. 
    We pick one node $r$ in the cycle to be the root and remove one incident edge $e=\{r,w\}$ of $r$, which is contained in the cycle, from the graph. Then we color the tree with root $r$ as described above. We additionally get $x(G_i,C,r)=1$. We reinsert the edge $e$ to the graph and color it such that $C(e)=C(w)$. This increases the number of incident colors of $r$ to two, and because $w$ only has one parent in the tree structure, we still have only two incident colors for $w$.
    \item[(d)] Suppose $x^*(G)\leq\left\lceil \frac{n}{2}\right\rceil.$
    Then it is possible to find a $x$-satisfying coloring for $x=\left\lceil \frac{n}{2}\right\rceil$. In such a coloring $n - \left\lceil \frac{n}{2}\right\rceil +1 = \left\lfloor \frac{n}{2}\right\rfloor + 1 $ incident edges have to be colored in $C(v)$ for each node $v$. Thus, the graph has at least $2n \left(\left\lfloor \frac{n}{2}\right\rfloor + 1 \right) > n^2$ edges, because $\left\lfloor \frac{n}{2}\right\rfloor + 1 > \frac{n}{2}$. This contradicts the fact, that $G$ has only $n^2$ edges.
Therefore, it can be concluded that $x^*(G)\geq\left\lceil \frac{n}{2}\right\rceil+1$. We define a valid star coloring with $x(G,C)=\left\lceil \frac{n}{2}\right\rceil+1$ using the following assignment. Let $G=(L \dot{\cup} R,E)$ be the bipartite graph with ascended numbered node sets $L=\{l_1,\hdots,l_n\}$ and $R=\{r_1,\hdots,r_n\}$. For each $l_i\in L$ with $i\leq \lfloor \frac{n}{2}\rfloor$ we direct the edge $\{l_i,r_j\}$ to $r_j$ if $j\leq \lfloor \frac{n}{2}\rfloor$ and to $l_i$ if $j > \lfloor \frac{n}{2}\rfloor$. 
For each $l_i\in L$ with $i > \lfloor \frac{n}{2}\rfloor$ we direct the edge $\{l_i,r_j\}$ to $l_i$ if $j\leq \lfloor \frac{n}{2}\rfloor$ and to $r_j$ if $j > \lfloor \frac{n}{2}\rfloor$.

$\hfill \square$
\end{enumerate}   
\end{proof}
\subsection{Pseudocode for the \texttt{MinimumStarColoring} algorithm}
Prior to presenting the pseudocode of the \texttt{MinimumStarColoring} algorithm, it is necessary to introduce additional notation.
For a partial coloring $C\in \mathcal{C}'(G)$, we define $N_{dif}^V(v)$ 
as the set of adjacent nodes $w$ to $v$ with $C(w)= C(e)$ for an edge $e$ containing both $v$ and $w$ and we define $N^E_{eq}(v)$ 
as the set of edges that are colored with the color $C(v)$. 
Additionally, $N^E_{none}(v)$ denotes all incident edges to $v$ which are uncolored.
\begin{align*}
    N^E_{eq}(v)\coloneqq & \{ {e \in  E'} \;|\; v \in e \text{ and } C(v)= C(e)\} \\
    N^E_{none}(v) \coloneq & \{ e \in E \; | \; v\in e, \ e \text{ is uncolored}\}\\
    N_{dif}^V(v)\coloneq & \{ w\in V \; | \; \exists \ {e\in E'}\text{ with }{v,w \in e} \text{ and }C(w)= C(e)\} 
\end{align*}
Note that $N_{dif}^V(v) \subseteq V$ and $N^E_{eq}(v),N^E_{none}(v) \subseteq E$ hold. 
\begin{algorithm}[h!]
    \small
    \DontPrintSemicolon

    \SetKwProg{Fn}{Function}{:}{}

\Fn{\ColorOneEdge{$v$}}{ 
    \KwData{Current node $v$}
    \KwResult{Returns TRUE if an additional edge incident to $v$ has been colored in place with $C(v)$ without changing $|N^E_{eq}(w)|$ for other nodes $w \neq v$, else return FALSE.}
    \If(\tcp*[f]{\scriptsize there is an uncolored edge incident to $v$}){$N^E_{none}(v)\neq \emptyset$}{ 
        choose $e \in N^E_{none}(v)$ \\
        $C(e) \leftarrow C(v)$\\
        \Return TRUE
    }
    \While(\tcp*[f]{\scriptsize every remaining neighbor that colors an incident edge}){$N_{dif}^V(v)\backslash K \neq \emptyset$}{
        choose $u\in N_{dif}^V(v)\backslash K$\\
         $K \leftarrow K \cup \{ u \}$ \tcp*[f]{\scriptsize store the asked nodes in $K$} \\
        \If{ \ColorOneEdge{$u$}}{
            let $e$ be the edge with $v,u \in e$\\
            $C(e) \leftarrow C(v)$\\
            \Return TRUE
        }
    }
    \Return FALSE 
}

\Fn{\EnsureFeasibility{}}{ 
\KwResult{Partial $\Delta(G)$-satisfying coloring $C$ of $G$}
\ForEach{$v\in V$ \textbf{with} $\kappa_v < \delta(v)$}{ 
\For{$i = 1 $ \textbf{to} $ \delta(v)- \kappa_v +1$}{
 $K \leftarrow \{ v \}$\\
 \If{$\neg$ \ColorOneEdge{$v$}}{
 \Return FALSE
 }
}
}
\Return TRUE
}

\Fn{\MinimumStarColoring{$G$}}{
\KwData{Nonempty, connected, simple graph $G$ with node capacities $\kappa_v$}
\KwResult{$x^*(G)$, partial $x^*(G)$-satisfying coloring $C$ of $G$. $x^*(G)= \infty$ if the instance is infeasible}
$C \leftarrow $ empty coloring,
$x \leftarrow \Delta(G)$, $K \leftarrow \emptyset$ \\
set $G$, $C$, $K$ and $x$ as global variables \\
\If{$\neg$ \EnsureFeasibility{}}{
\Return $\infty , C$ \tcp*[f]{\scriptsize infeasible}
}
$x \leftarrow \Delta(G)-1$ \\
\While{$x >0$}{
\ForEach{$v \in V$ \textbf{with} $|N^E_{eq}(v)|<l(v,x) \ \wedge  \ l(v,x)>1$}{
\While{$|N^E_{eq}(v)|<l(v,x)$}{
 $K \leftarrow \{ v \}$\\
 \If{ $\neg$ \ColorOneEdge{$v$} }{
 \Return $x+1, C$ \tcp*[f]{\scriptsize $C$ is $x+1$-satisfying and a coloring with $x$ is not possible }
 }
}
}
$x \leftarrow x - 1$
}
\Return $x+1 , C$

}
\caption{The \texttt{MinimumStarColoring} algorithm, which computes the minimum star coloring number $x^*(G)$ and a valid partial $x^*(G)$-satisfying coloring $C$ of $G$.}
\label{algorithm}
\end{algorithm}

\clearpage
\subsection{Additions to general hypergraphs of Section \ref{section:hypergraphs}}
For non-linear hypergraphs the coloring algorithm may not work. In Figure \ref{fig:hypergraph-counter-example} we visualized an example where our algorithm does not yield the correct solution. 
It is possible that two different edges incident to a node $v$ that may get the same color do not get the color $C(v)$. So, nodes can have incident edges with the same color. 
We always assumed that the only way to reduce the number of still possible incident colors for edges incident to a node $v$ is by letting $v$ color more incident edges in $C(v)$. This may be violated in this case. 
It is still an open problem whether finding a minimum star coloring of general hypergraphs is in \texttt{P}. We assume that the decision version of this problem is actually \texttt{NP-hard}.
\begin{figure}[htb]
    \centering
    \subfloat[]{
\begin{tikzpicture}[scale=1,
block/.style={
      rectangle,
      draw,
      thick,
      align=center,
      rounded corners
    }
]
    \foreach \i in {1,2,3,4,5}{
    \node[vertex,fill=c\i!50] (v\i) at (\i,0){$v_{\i}$};
    }
\node[block,minimum height=0.8cm,color=c3,fit=(v1) (v3)] {};
\node[block,minimum height=1cm,color=c3,fit=(v2) (v4)] {};
\node[block,minimum height=1.2cm,color=c3,fit=(v3) (v5)] {};
\end{tikzpicture}
}
\subfloat[]{
\hspace{1cm}
\begin{tikzpicture}[scale=1,
block/.style={
      rectangle,
      draw,
      thick,
      align=center,
      rounded corners
    }
]
    \foreach \i in {1,2,3,4,5}{
    \node[vertex,fill=c\i!50] (v\i) at (\i,0){$v_{\i}$};
    }
\node[block,minimum height=0.8cm,color=c3,fit=(v1) (v3)] {};
\node[block,minimum height=1cm,color=c2,fit=(v2) (v4)] {};
\node[block,minimum height=1.2cm,color=c3,fit=(v3) (v5)] {};
\end{tikzpicture}
}
    \caption{Example of a general hypergraph for which the \texttt{MinimumStarColoring} algorithm does not work. (a) A valid coloring with $x^*(G)=1$. 
    The algorithm starts with $x=2$. The only node $v$ with $l(v,x)>1$ is $v_3$ with $l(v_3,x)=2$.
    Thus, we have to color two different edges $v_3$ is contained in, e.g. $\{v_1,v_2,v_3\}$ and $\{v_3,v_4,v_5\}$. We reduce $x$ to $x=1$. Now the algorithm tells us to color two edges for the nodes $v_2$ and $v_4$, and one additional edge for $v_3$. If we start with $v_2$, the depth-first search is only able to color the edge $\{v_2,v_3,v_4\}$ with the color of $v_2$. Thus, the algorithm returns the $2$-satisfying coloring shown in (b), which is not optimal.} 
    \label{fig:hypergraph-counter-example}
\end{figure}
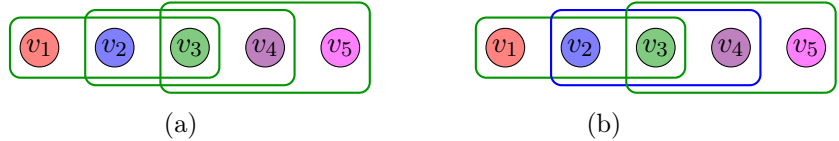

\subsection{Steps of the improved algorithm of Section \ref{ch:improvedAlgo}}\label{app:improvedAlgo}

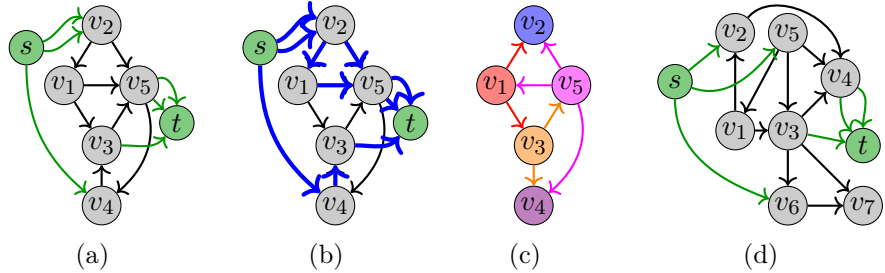
\begin{figure}[htb]
\centering
\subfloat[]{
\begin{tikzpicture}[scale=1]
\tikzset{%
    vertex/.style={%
        circle,inner sep=1,fill=gray!40, draw=black, minimum width=13pt
    }
}
\coordinate (c1) at (0.5,0);
\coordinate (c2) at (1,0.8);
\coordinate (c4) at (1,-1.6);
\coordinate (c3) at (1,-0.8);
\coordinate (c5) at (1.5,0);
\coordinate (cs) at (0,0.5);
\coordinate (ct) at (2,-0.5);
\foreach \i in {1,2,3,4,5} {
\node[vertex] (v\i) at (c\i) {$v_{\i}$};
}
\foreach \i in {s,t} {
\node[vertex,fill=c3!50] (v\i) at (c\i) {$\i$};
}

\foreach \i/\j in {2/1,1/3,1/5,4/3,2/5,3/5}{
\draw [->,thick] (v\i) to (v\j);
}

\draw [->,thick] (v5) to[out=-70,in=40]  (v4);

\draw [->,thick,c3] (vs) to[out=20,in=160] (v2);
\draw [->,thick,c3] (vs) to[out=-10,in=-160] (v2);
\draw [->,thick,c3] (vs) to[out=-90,in=140] (v4);

\draw [->,thick,c3] (v5) to[out=20,in=90] (vt);
\draw [->,thick,c3] (v5) to (vt);
\draw [->,thick,c3] (v3) to[out=0,in=-130] (vt);

\end{tikzpicture}
} \quad
\subfloat[]{
\begin{tikzpicture}[scale=1]
\tikzset{%
    vertex/.style={%
        circle,inner sep=1,fill=gray!40, draw=black, minimum width=13pt
    }
}
\coordinate (c1) at (0.5,0);
\coordinate (c2) at (1,0.8);
\coordinate (c4) at (1,-1.6);
\coordinate (c3) at (1,-0.8);
\coordinate (c5) at (1.5,0);
\coordinate (cs) at (0,0.5);
\coordinate (ct) at (2,-0.5);
\foreach \i in {1,2,3,4,5} {
\node[vertex] (v\i) at (c\i) {$v_{\i}$};
}
\foreach \i in {s,t} {
\node[vertex,fill=c3!50] (v\i) at (c\i) {$\i$};
}

\foreach \i/\j in {2/1,1/3,1/5,4/3,2/5,3/5}{
\draw [->,thick] (v\i) to (v\j);
}
\draw [->,thick] (v5) to[out=-70,in=40]  (v4);

\draw [->,ultra thick,c2] (vs) to[out=20,in=160] (v2);
\draw [->,ultra thick,c2] (vs) to[out=-10,in=-160] (v2);
\draw [->,ultra thick,c2] (vs) to[out=-90,in=140] (v4);

\draw [->,ultra thick,c2] (v5) to[out=20,in=90] (vt);
\draw [->,ultra thick,c2] (v5) to (vt);
\draw [->,ultra thick,c2] (v3) to[out=0,in=-130] (vt);

\foreach \i/\j in {2/1,2/5,1/5,4/3}{
\draw [->,ultra thick,c2] (v\i) to (v\j);
}
\end{tikzpicture}
}\quad
\subfloat[]{
\begin{tikzpicture}[scale=1]
\tikzset{%
    vertex/.style={%
        circle,inner sep=1,fill=gray!40, draw=black, minimum width=13pt
    }
}
\coordinate (c1) at (0.5,0);
\coordinate (c2) at (1,0.8);
\coordinate (c4) at (1,-1.6);
\coordinate (c3) at (1,-0.8);
\coordinate (c5) at (1.5,0);
\coordinate (cs) at (0,0.5);
\coordinate (ct) at (2,-0.5);

\foreach \i in {1,2,4,5} {
\node[vertex,fill=c\i!50] (v\i) at (c\i) {$v_{\i}$};
}
\node[vertex,fill=c6!50] (v3) at (c3) {$v_3$};

\foreach \i/\j in {1/2,1/3,5/1,5/2}{
\draw [->,thick,color=c\i] (v\i) to (v\j);
}
\draw [->,thick,color=c5] (v5) to[out=-70,in=40]  (v4);

\foreach \i/\j in {3/4,3/5}{
\draw [->,thick,color=c6] (v\i) to (v\j);
}

\foreach \i/\j in {}{
\draw [->,thick,c1] (v\i) to (v\j);
}
\end{tikzpicture}
}\quad \;
\subfloat[]{
\begin{tikzpicture}[scale=1]
\tikzset{%
    vertex/.style={%
        circle,inner sep=1,fill=gray!40, draw=black, minimum width=13pt
    }
}
\coordinate (c1) at (0.3,-1);
\coordinate (c2) at (0.3,0.3);
\coordinate (c3) at (1,-1);
\coordinate (c4) at (1.7,-0.3);
\coordinate (c5) at (1,0.3);
\coordinate (c6) at (1,-2);
\coordinate (c7) at (2,-2);
\coordinate (cs) at (-0.5,-0.35);
\coordinate (ct) at (2,-1.2);

\foreach \i in {1,2,3,4,5,6,7} {
\node[vertex] (v\i) at (c\i) {$v_{\i}$};
}
\foreach \i in {s,t} {
\node[vertex,fill=c3!50] (v\i) at (c\i) {$\i$};
}

\foreach \i/\j in {1/3,3/4,5/1,5/3,5/4,3/6,3/7,6/7}{
\draw [->,thick] (v\i) to (v\j);
}
\draw [->,thick] (v1) --(v2);
\draw [->,thick] (v2) to[out=50,in=90]  (v4);

\draw [->,thick,c3] (vs) to (v2);
\draw [->,thick,c3] (vs) to[out=-20,in=-130] (v5);
\draw [->,thick,c3] (vs) to[out=-90,in=160] (v6);

\draw [->,thick,c3] (v3) to (vt);
\draw [->,thick,c3] (v4) to[out=-30,in=90] (vt);
\draw [->,thick,c3] (v4) to[out=-90,in=130] (vt);

\end{tikzpicture}
}
\caption{In (a) the flow multigraph $\Tilde{G}$ is constructed for $x=2$ without node capacities. In (b) an integer maximum flow in $\Tilde{G}$ is computed. Reversing the edges along the flow and removing $s$ and $t$ results in an orientation, where the induced coloring $C$ fulfills $x(G, C)=2$, which is visualized in (c). In (d) we try to find an orientation in $G'$, such that the induced coloring $C$ fulfills $x(G', C)=2$. Each maximum $s-t$ flow in $\Tilde{G'}$ has a value, which is smaller than the indegree of $t$, thus there exists no valid star coloring of $G'$ with $x(G', C)=2$.}
\label{fig:example_improvedAlgo}
\end{figure} 
Applying the improved algorithm to an example graph where it is possible to improve the objective value and how it is done is visualized in Figure \ref{fig:example_improvedAlgo}(a-c). 
In Figure \ref{fig:example_improvedAlgo}(d) we show a multigraph $\Tilde{G}$ constructed with $x=2$ where an integer flow using every edge incident to $t$ does not exist.

\subsection{Omitted proof of Section \ref{sec:StarInd}}\label{app:secReductions}
\reductionMMIPTOMMSP*
\noindent
\begin{proof}\renewcommand{\qed}{}
Given a graph $G=(V,E)$ with node capacities $\kappa$, for which we want to solve \textsc{MinMaxInd}, we construct a graph $G'$ with node capacities $\kappa'$, on which we solve \textsc{MinMaxStar} instead.

The graph $G'=(V \cup V', E \cup E')$ is constructed as follows: For each node $v$, we insert a copy $v'$ into $V'$ and insert the edge $\{v,v'\}$ into $E'$. The node $v$ gets the capacity $\kappa'_v := \kappa_v+1$ and $v'$ has the capacity $\kappa'_{v'}=1$. This construction is visualized in Figure \ref{fig:reduction_MMIP_to_MMSP}. We prove that $k^*(G)=x^*(G')-1$:

Let $k \coloneqq k^*(G)<\infty$. There exists an edge orientation such that for each node $v \in V$ the maximum indegree is at most $k$. So, we construct the star $S_v$ and insert all outgoing edges from $v$ into the star. Additionally, we insert the edge $\{v,v'\}$ into $S_v$. So, $v$ is contained in the star $S_v$. Each incoming edge at $v$ belongs to a different star. Since the edge orientation is valid, there are at most $k$ incoming edges. Therefore, $v$ is contained in at most $k+1$ stars. Each node $v' \in V'$ is only contained in the star $S_{v'}$. Thus $x^*(G')-1 \leq k^*(G)$.

Now let $x\coloneqq x^*(G')<\infty$. We inspect the star partitioning, where each node is contained in at most $x$ stars. The edge $\{v,v'\}$ belongs to exactly one star. If $\{v,v'\}$ belongs to the star $S_{v'}$, we delete $S_{v'}$ and add $\{v,v'\}$ to the star $S_v$. The node $v'$ was contained in $S_{v'}$ and is now in $S_v$ instead. If the star $S_v$ already existed before, the number of stars $v$ is contained in decreases by one, because $v$ is no longer in $S_{v'}$. If the star $S_v$ did not exist before, $v$ is contained in the new star $S_v$ and deleted from the old star $S_{v'}$, so the number of stars $v$ is contained in does not change. Therefore, in this modification, each node is still contained in at most $x$ stars. Now each node $v\in V$ is contained in its own star $S_v$. Each edge $\{v,w\}$ is contained in exactly one star $S_v$, so we orientate it away from $v$.
For each $v \in V$, the number of stars other than $S_v$ in which the node $v$ is contained is exactly the indegree of $v$. Therefore $x^*(G')-1 \geq k^*(G)$. 

As illustrated above, the feasibility of both scenarios is equivalent, and thus $k^*(G)=\infty$ iff $x^*(G')=\infty$.$\hfill \square$
\end{proof}
\clearpage
\subsection{Omitted proofs of Section \ref{sec:redwIndwStar}}\label{app:secReductions2}
\noindent In order to prove Theorem \ref{thm:reductionWeightedCases} we first prove some useful properties of the gadget $G_v$ in the following lemma.
\begin{lemma}\label{lemma:orientation}
There is at least one valid star partition $C$ with \mbox{$x(G_v,C)\leq M+k$}.
In each valid star partition $C$ with $x(G_v,C)\leq M+k$, the following holds:
\begin{enumerate}
    \item[(a)] The edge $\{v,\Tilde{v}\}$ is directed to $\Tilde{v}$.
    \item[(b)] The edge $\{v,v_1\}$ is directed to $v$.
\end{enumerate}
\end{lemma}
\noindent\begin{proof}
Figure \ref{fig:reductionWeightedVariants} yields the existence, since $x(v_1)=x(v_2)=x(v_4)= k + M$ and $x(v_3)= 2 k + 2 w_v \leq k + M$. 
Every edge connected to a node $\Tilde{v}'\in \Tilde{V_{v}}$ is directed toward this node since $w_{\Tilde{v}'}> M+k$. This proves (a). 
Assume $\{v, v_1\}$ is directed toward $v_1$. Then $\{v_1, v_2\}$ is directed toward $v_2$, since otherwise $x(v_1)= w_v + (M - w_v ) + k+ w_v> M + k$. 
Similarly, $v_2$ can not have another incoming edge, since otherwise also $x(v_2)>M + k$. Thus, the edges $\{v_2, v_3\}$ and $\{v_2, v_4\}$ are orientated toward $v_3$ and $v_4$ respectively. 
Now we have two remaining cases:
If $\{v_3, v_4\}$ is directed toward $v_{3}$ we have $x(v_3)= (k+w_v) + (k+ w_v) + (M-w_v)>M + k$ and if it is directed toward $v_4$ we have $x(v_4)= (k+w_v) + (k+ w_v) + (M-w_v)>M + k$. This is a contradiction; thus, the edge $\{v,v_1\}$ is directed to $v$.
\end{proof}

\reductionWeightedCases*
\noindent
\begin{proof}\renewcommand{\qed}{}
For each graph $G=(V,E)$ for \textsc{W-MinMaxInd} we construct the graph $G'=(V',E')$.
If there exists an orientation $\Lambda$ of $G$ with a maximum weighted indegree lesser than or equal to $k$, then we can construct an orientation $\Lambda'$ of $G'$ corresponding to a valid star coloring $C$ with $x(G',C)\leq M+k$ with $M= k + 2 \max_{v \in V} w_v$ the following way:
For each edge $e \in E$ we set $\Lambda(e)=\Lambda'(e)$. And each edge $e \in E' \setminus E$ is in a gadget $G_v$. Thus, we orient each gadget $G_v$ as visualized in Figure \ref{fig:reductionWeightedVariants}. For each node in $v'\in V'\setminus V$ we proved in Lemma \ref{lemma:orientation}, that it fulfills the condition $x(G',C,v') \leq M+k$. Every node $v \in V$ is connected to edges in $E$ and edges in $E' \setminus E$. Additionally, $v$ has at least one outgoing edge $\{v,\Tilde{v}\}$. 
$$x(G',C,v)=  w_{v} + w_{v_1} +\sum_{\substack{u \in V: \\ \Lambda(\{u,v\})=(u,v)}} w_u  \leq w_v + (M- w_v) + k = M +k$$
Let $C$ be the coloring of a valid star partition of $G'$ with $x(G',C) \leq M+k$. We define $\Lambda(e)=\Lambda'(e)$ for each $e \in E$. With Lemma \ref{lemma:orientation} we know for every $v\in V$ the orientation of the edges $\{v, \Tilde{v}\}$ and $\{v,v_1\}$, which results in a value of $x(G_v,C,v)=M$. Thus, the sum of the incoming edges from $E$ is at most $k$. 

The construction of $G'$ can be done in linear time since we deterministically add only $9 |V|$ nodes and $10 |V|$ edges with weights, which linearly depend on $k$ and the weights of $G$.
$\hfill \square$
\end{proof}

\end{document}